\documentclass[final]{IEEEtran}

\IEEEoverridecommandlockouts              
\usepackage{graphicx}
\usepackage{url}
\usepackage{amsmath,amssymb}
\usepackage{amsthm,thmtools}
\usepackage{amsfonts} 
\usepackage{enumerate}
\usepackage{color}           
\usepackage{verbatim}
\usepackage{amssymb}
\usepackage{amsbsy}
\usepackage{amsmath}
\usepackage{setspace}
\usepackage{url}
\usepackage{float}
\usepackage{dsfont}
\usepackage{siunitx}
\usepackage{epsfig}
\usepackage{url}

\newtheorem{theorem}{Theorem}
\newtheorem{lemma}{Lemma}

\newtheorem{proposition}{Proposition}

\newtheorem{assumption}{Assumption}

\newcommand{\be}{\begin{equation}}
\newcommand{\ee}{\end{equation}}
\newcommand{\ben}{\begin{equation*}}
\newcommand{\een}{\end{equation*}}

\begin{document}
	\title{A New Approach for Feedback Stabilization and its\\
		Application for Data-Driven Control of Polynomial Systems}
	\author{Diego~de~S.~Madeira,~Jo\~ao~Gabriel~N.~Silva,~Wilkley B. Correia,~Antonis Papachristodoulou 
		\thanks{Diego de S. Madeira, Jo\~ao Gabriel N. Silva and Wilkley B. Correia are with the Electrical Engineering Department, Federal University of Cear\'a (UFC), Fortaleza 60455-760, Brazil (e-mails: dmadeira@dee.ufc.br, joao.gns97@alu.ufc.br, wilkley@dee.ufc.br). This research was developed while Diego de S. Madeira was a visiting scholar at the Department of Engineering Science, University of Oxford, OX1 3PJ Oxford, United Kingdom. This work was supported by he National Council for Scientific and Technological Development (CNPq), Brazil, under Grant 402731/2023-9 and Grant 442110/2023-5.
		}
		\thanks{Antonis Papachristodoulou is with Engineering Science, University of Oxford, OX1 3PJ Oxford, UK. (e-mail: antonis@eng.ox.ac.uk). AP's work was supported by EPSRC Projects EP/Y014073/1 and UKRI2108. 
		}
	}
	
	\maketitle
	
	\begin{abstract}    
		Inspired by recent developments in dissipativity-based control, this work proposes new sufficient conditions for asymptotic stabilization of nonlinear systems using state feedback. We prove that this new framework is well suited for data-driven control of polynomial systems using noisy measurements, a topic that has lately attracted considerable attention. Two iterative procedures for data-based state feedback design are provided using the proposed framework, that use sum-of-squares (SOS) optimization. Typical limitations of conventional SOS methods based on alternating (D-K) procedures for controller design, such as the need to provide an initialization for a control-Lyapunov function (CLF), are overcome in this paper. Numerical examples demonstrate the applicability and the advantages of the new strategies. 
	\end{abstract}
	
	
	\section{Introduction}
	
	Within nonlinear systems theory, the field of data-based controller design has witnessed a surge of interest in recent years \cite{hjal},  \cite{baza1}. Various strategies for solving the problem of feedback stabilization of dynamical systems that are not precisely known, but from which a sufficiently rich set of measurements is available, have been proposed \cite{waarde1}, \cite{willm3}, \cite{morari1}. Controller design for polynomial models, in particular, is of special interest in this context \cite{prajna1}. Motivated by the fact that polynomial state-space representations may accurately approximate the behaviour of nonlinear systems in a region, and also favoured by the development of sum-of-squares (SOS) optimization  \cite{parri1}, the study of polynomial models became a rich research subject \cite{lasserre}. Soon after some recent breakthroughs in data-driven control were reported for linear-time invariant (LTI) systems \cite{persis1}, \cite{waarde1}, extensions to the polynomial case have followed. 
	
	Global stabilization of polynomial systems using noisy data and polynomial state feedback was addressed in \cite{persis2}, where a quadratic-like structure for the control-Lyapunov function (CLF) was considered, as well as a certain decomposition for the control law itself. A new approach for state feedback stabilization using Petersen's lemma was derived in \cite{biso1}, without imposing those assumptions on the CLF and the control law, thus generalizing the results of \cite{persis2}. Instead, an alternating procedure, a so-called D-K iterative approach was employed to enhance performance, and an initialization for the CLF was obtained via model linearization. Saturated state feedback design was investigated in \cite{wilkley2}, using a non-iterative dissipativity-based framework. Safe control was addressed in \cite{persis3}, where a D-K strategy was applied again and an initialization was also obtained by model linearization. An iterative and dissipativity-based strategy for controller design was developed in \cite{joao1}, using Petersen's lemma. Although D-K strategies and assumptions on the CLF were avoided, a fictitious output variable needed to be specified in advance.
	
	\emph{Contributions:} Here we present new sufficient conditions for closed-loop stabilization of nonlinear systems, inspired by the dissipativity-based framework for controller design from \cite{made1}. In contrast to the strategy discussed in \cite{made1}, which allows to determine \emph{a set of stabilizing control laws} by satisfying a dissipation inequality, our new approach aims to  determine a \emph{set of closed-loop dynamics} compatible with a candidate CLF. We provide substantial evidence of the natural applicability of this new framework for data-driven control of polynomial systems. Notably, it is proven to be fully compatible with the well-known representation of \emph{feasible systems} through matrix ellipsoids used in \cite{biso1}. But unlike \cite{biso1}, alternating (D-K) iterations for controller design are avoided here, as well as the need to provide any initializations for either a CLF or a control law. 
	
	In addition, our work also differs from the one in \cite{berbe1}. Firstly, we consider polynomial continuous-time systems, instead of the discrete-time linear systems addressed in \cite{berbe1}. Furthermore, no parametrization of the closed-loop dynamics is proposed in this work. Instead, we propose two SOS iterative strategies for data-driven control. In the first one, each iteration is of lower computation complexity when compared to other SOS approaches in the field, and it is based on arguments that resemble those of robust control theory. With regards to the second SOS algorithm, a stabilizing controller is usually obtained in the very first few iterations, although each repetition is computationally more demanding than in the first method. The new framework reveals fundamental relationships between robust Lyapunov conditions and the data sets obtained through experiments.     
	
	In Section II, the notation employed throughout this paper is introduced and, in Section III, a set of preliminary statements relevant to this work are presented. In Section IV, a new approach for state feedback stabilization is presented. Section V contains new data-driven conditions and algorithms for controller design using the results of Section IV. Examples are given in Section VI, and Section VII contains our concluding remarks and future research directions.

	\section{Notation}      
	
	$\mathbb{R}$ denotes the set of real numbers and $\mathbb{N}_{>0}$ all positive natural numbers. $\mathbb{R}_{\ge 0}$ and $\mathbb{R}_{>0}$ are, respectively, the set of all $\beta \in \mathbb{R}$ such that $\beta \ge 0$, $\beta > 0$. $\mathbb{R}^n$ is the set of real column vectors and $\mathbb{R}^{n \times m}$ denotes all $n \times m$ real matrices. $M^\top$ is the transpose of a matrix $M$, and $M^{-1}$ is its inverse. $I_n$ is the identity matrix of  dimension $n$, and $x_i$ is the $i_{th}$ element of a vector $x \in \mathbb{R}^n$. $\mathbb{S}^n$ is the set of symmetric $n \times n$ matrices. $M\succ 0$ ($\succeq 0$) means that $M\in \mathbb{S}^n$ is positive definite (semidefinite), and $M\prec 0$ ($\preceq 0$) means that $M$ is negative definite (semidefinite).	$||\cdot||$ is the Euclidean norm of a vector. $f:\mathcal{X} \rightarrow \mathcal{Y}$ refers to a function $f$, a domain $\mathcal{X}$ and a codomain $\mathcal{Y}$. $\mathcal{X}\times \mathcal{Y}$ is the Cartesian product of sets $\mathcal{X}$ and $\mathcal{Y}$. $f:\mathcal{X} \rightarrow \mathbb{R}$ is a positive definite (semidefinite) function, referred to as $f(x)\succ 0$ ($\succeq 0$), if $f(0)=0$ and $f(x)\in \mathbb{R}_{>0}$ $(\mathbb{R}_{\ge 0})$ for any $x \in \mathcal{X}$, $x \neq 0$. $\mathcal{P}[x]$ stands for the polynomials on $x$, and $\mathcal{P}^{n\times m} [x]$ is the ring of the polynomial matrices on $x$ of dimension $n\times m$. $\sum [x]$ stands for the set of SOS polynomials on $x$, that is, the set of all $p(x)\in\mathcal{P}[x]$ for which there exist $c_j (x)\in\mathcal{P}[x], j = 1,2,\ldots,k,$ such that $p(x)=\sum_{j=1}^{k} (c_j (x))^2$. Likewise, $\sum^n [x]$ is the ring of the sum-of-squares $n \times n$ matrix polynomials on $x$. For some continuously differentiable $\phi : \mathbb{R}^n \rightarrow \mathbb{R}$, $\nabla \phi (x):= [\frac{\partial \phi}{\partial x_1}~\frac{\partial \phi}{\partial x_2} \dots \frac{\partial \phi}{\partial x_n}]^\top$. For given $V(x)\succeq 0$ and $\rho \in \mathbb{R}_{>0}$ , $\mathcal{E}(V,\rho):=\left\{x \in \mathbb{R}^n ~|~V(x) \le \rho \right\}$. The symbol $\star$ stands for a symmetric block in a matrix.

	\section{Preliminary Definitions}
	
	Input-affine polynomial systems with the following state-space representation are considered in this work
	\begin{align}\label{eq:c5_1}
		\dot{x}(t) = f(x(t)) + g(x(t)) u(t),
	\end{align}
	where $x(t) \in \mathcal{X} \subseteq \mathbb{R}^n$ is the state vector and $u(t)\in \mathcal{U} \subseteq \mathbb{R}^m$ is the control input, $\forall t\in \mathbb{R}_{\ge 0}$. Let $f:\mathcal{X} \rightarrow \mathbb{R}^n$ and also $g:\mathcal{X} \rightarrow \mathbb{R}^{n\times m}$ be polynomial functions that satisfy the following assumption \cite{persis2}:
	
	\begin{assumption}\label{ass:a_1} For the polynomial nonlinear system (\ref{eq:c5_1}):
		\begin{itemize}
			\item $f(0)=0$, i.e., the origin is a known equilibrium of the plant.
			\item One can determine an upper bound for the maximum degree $\delta_{fg}$ on $x$ of functions $f(x)$ and $g(x)$.
		\end{itemize}
	\end{assumption}

	Assumption \ref{ass:a_1} allows to present the state-space model from (\ref{eq:c5_1}) in terms of a linear-like representation such as
	\begin{align}\label{eq:sos_2}
		\dot{x} = AZ(x) + BW(x) u,
	\end{align}
	where constant matrices $A\in\mathbb{R}^{n\times N}$ and $B\in\mathbb{R}^{n\times q}$ that verify $f(x)=AZ(x)$ and $g(x)=BW(x)$ are guaranteed to exist, but remain \emph{unknown}. The vector function $Z(x)\in \mathcal{P}^{N\times 1} [x]$, however, is \emph{known} and collects all distinct monomials on $x$ that appear in $f(x)$, and it satisfies $Z(0)=0$. The state-dependent matrix $W(x)\in \mathcal{P}^{q\times m} [x]$ is also \emph{known} and contains all monomials of $x$ in $g(x)$.  
	
	In order to be able to design a stabilizing controller using data, we first run an experiment over a time interval given by $[t_{I}, t_{F}]$, and collect a set of input-state samples at certain time instants $\{ t_0 ,t_1 ,\cdots , t_{T-1} \}$, where $t_0\ge t_{I}$, $t_{T-1} \le t_{F}$, and $T\in\mathbb{N}_{>0}$ is the number of samples. Then, a set of matrices containing the collected data is conveniently defined as  
	\begin{align}\label{eq:sos_4}
		U_0 := [u(t_0 )~~u(t_1)~~\cdots ~~u(t_{T-1})],\nonumber\\
		X_0 := [x(t_0 )~~x(t_1)~~\cdots ~~x(t_{T-1})],\\
		\hat{X}_1 := [\dot{x}(t_0 )~~\dot{x}(t_1 )~~\cdots ~~\dot{x}(t_{T-1})],\nonumber
	\end{align}
	such that the following data-based matrices can be computed
	\begin{align}\label{eq:sos_5}
		Z_0 := [Z(x(t_0 ))~~Z(x(t_1))~~\cdots ~~Z(x(t_{T-1}))],\nonumber\\
		\bar{U}_0 := [W(x(t_0 ))u(t_0 )~~W(x(t_1))u(t_1)~~~~~~~~~~~~~~~~\nonumber\\
		\cdots ~~W(x(t_{T-1}))u(t_{T-1})].
	\end{align}
	
	As in \cite{persis2}, we assume that only the time derivatives $\dot{x}(t)$ of the state vector are affected by measurement noise, while $Z_0$ and $\bar{U}_0$ are determined using the exact values of $x(t)$ and $u(t)$. Next, by considering $d(t)\in \mathbb{R}^n$ as the noise present on the samples of $\dot{x}(t)$, let us define the \emph{unknown} matrix
	\begin{align}
		D_0 := [d(t_0 )~~d(t_1)~~\cdots ~~d(t_{T-1})],\nonumber
	\end{align}
	and suppose that a matrix $X_1 := \hat{X}_1 +D_0$ contains the noisy measurements of $\dot{x}(t)$. Note that, in view of (\ref{eq:sos_2}), the following relation on the data holds
	\begin{align}\label{eq:sos_8}
		X_1 = AZ_0 +B\bar{U}_0 +D_0 ,
	\end{align}
	and that in realistic scenarios, $D_0$ is bounded. Hence, we consider the following assumption on $D_0$ \cite{persis2}
	
	\begin{assumption}\label{ass:a_2} The unknown matrix $D_0$ satisfies $D_0 D^\top_0 \preceq R_D R^\top_D$ for some
		known $R_D \in \mathbb{R}^{n\times T}$.
	\end{assumption}
	
	Assumption \ref{ass:a_2} is common in the field and the set of matrices $[A~B]$ consistent with that level of noise on the measurements of $\dot{x}(t)$ can be presented as in Bisoffi \emph{et. al.} \cite{biso1}
	\begin{equation}\label{eq:ox_14}
		\mathcal{C}:= \{ [A~B]: X_1 =AZ_0 +B\bar{U}_0 + D,~D\in\mathcal{D} \},
	\end{equation}
	where 
	\begin{equation}
		\mathcal{D}:= \{ D\in \mathbb{R}^{n\times T} : DD^\top \preceq R_D R^\top_D\}.\nonumber
	\end{equation}
	One can argue that the set in (\ref{eq:ox_14}) is equivalently given by 
	\begin{equation}\label{eq:ox_17}
		\mathcal{C}:= \Bigg\{ [A~B]: [I~A~B]\cdot\left[ \begin{array}{c|c}
			\textbf{C} & \textbf{B}^\top \\ \hline
			\textbf{B} & \textbf{A}
		\end{array} \right] \cdot [\star]^\top \preceq 0\Bigg\},
	\end{equation}
	with
	{\small
		\begin{align}\label{eq:ox_18}
			\left[ \begin{array}{c|c}
				\textbf{C} & \textbf{B}^\top \\ \hline
				\textbf{B} & \textbf{A}
			\end{array} \right] = 
			\left[ \begin{array}{c|c}
				X_1 X^\top_1 - R_D R^\top_D & \star \\ \hline \\
				-\begin{bmatrix} Z_0\\ \bar{U}_0 \end{bmatrix} X^\top_1 & 
				\begin{bmatrix} Z_0\\ \bar{U}_0 \end{bmatrix} \begin{bmatrix} Z_0\\ \bar{U}_0 \end{bmatrix}^\top
			\end{array} \right].
	\end{align}}
	
	Finally, a further and equivalent presentation of $\mathcal{C}$ is as follows
	\begin{equation}\label{eq:ox_78}
		\mathcal{C}:= \Big\{ [A~B]=Z^\top_{AB}: (Z_{AB}-\xi)^\top \textbf{A} (Z_{AB}- \xi) \preceq \textbf{Q} \Big\},
	\end{equation}
	where from \cite[Lem. 3 and Asm. 1]{biso1}, one should consider that $\textbf{A}\succ 0$, 
	\begin{equation}\label{eq:ox_126}
		\xi=-\textbf{A}^{-1}\textbf{B},~~\textbf{Q} = \textbf{B}^\top \textbf{A}^{-1} \textbf{B}-\textbf{C}\succeq 0.
	\end{equation}
	Conditions (\ref{eq:ox_14}), (\ref{eq:ox_17}) and (\ref{eq:ox_78}) capture the set of dynamics compatible with the data and have been employed in the literature to derive sufficient conditions for data-driven control of nonlinear systems \cite{persis2} -- \cite{persis3}. The preliminary results discussed in this section allow to formulate the main problem addressed in this work, where we consider that the domain of attraction (DOA) $\mathcal{X}_{doa}$ is the set of all $x(0)\in \mathcal{X}$ such that, if $x(0)\in\mathcal{X}_{doa}$, then $lim_{t\rightarrow \infty} x(t)=0$ \cite{khali1}.
	\\$ $\\
	\emph{Problem statement:} Design a polynomial state feedback $u(x)$ that asymptotically stabilizes nonlinear system (\ref{eq:sos_2}),  $\forall [A~B]\in \mathcal{C}$, and estimate a domain of attraction for the closed-loop system.

	\section{A New Approach for State Feedback Design}
	\label{sec:c5_8}
	
	In this work, the dissipativity-based conditions for feedback stabilization of precisely known systems presented in \cite{made1} are adapted to solve the problem of data-driven controller design. In that reference, the relation below was applied to obtain a control law 
	\begin{align}\label{eq:c1_7}
		\dot{V}(x)+T(x) \le h(x)^\top Qh(x)+2h(x)^\top Su+u^\top Ru, 
	\end{align}
	where $h(x)$ is an output variable, $V:\mathcal{X} \rightarrow \mathbb{R}$, $V(x)\succeq 0$, $T:\mathcal{X} \rightarrow \mathbb{R}$, $T(x)\succ 0$, and $Q\in \mathbb{S}^{p}$, $S\in \mathbb{R}^{p \times m}$ and $R\in \mathbb{S}^m$ are constant matrices. As discussed in \cite{made12} and \cite{wilkley2}, a conventional dissipation inequality such as (\ref{eq:c1_7}) can be employed to determine a \emph{set of control laws} that stabilize a given system and, as a result, it can even be applied to saturating feedback design. In this section, though, using a modified version of (\ref{eq:c1_7}), we aim to determine a \emph{set of closed-loop dynamics} that are compatible with a certain CLF, and thus stable. We prove that this approach is rather meaningful to data-driven control, where the underlying model may not be precisely known, but one must be able to fulfill a closed-loop stability condition nevertheless. In other words, one needs to derive the specific conditions which, if feasible, guarantee that a single CLF applies to all plants compatible with the data and to a certain control law \cite{biso1}. 
	
	Consider then the following theorem, where a new approach for closed-stabilization of nonlinear models is proposed. We assume for now that $f(x)$ and $g(x)$ in (\ref{eq:c5_1}) are known, and later on consider the case of unknown systems. In addition, let a contractive invariant set $\mathcal{X}_c $ be the set of all $x(0)\in\mathcal{X}$ such that, if $x(0)\in\mathcal{X}_c$, then $x(t) \in \mathcal{X}_c$ for all $t\in\mathbb{R}_{\ge 0}$ and $\lim_{t\to \infty} x(t) = 0$ \cite{khali1}.  
	
	\begin{theorem}\label{thm:thm_3} Suppose that the following inequality holds for system (\ref{eq:c5_1}), $\forall (x,u_e) \in \mathcal{X} \times \mathbb{R}^n$:
		\begin{align}\label{eq:ox_4}
			\nabla V(x)^\top u_e + T(x) \le \mathcal{Q}(x) ~~~~~~~~~~~~~~~~~~~~~~~~~~~\nonumber\\
			- 2[f(x)+g(x)u(x)]^\top \mathcal{R} u_e + u_e^\top \mathcal{R} u_e,
		\end{align}	
		where the decision variables are such that $V:\mathcal{X} \rightarrow \mathbb{R}$ and $T:\mathcal{X} \rightarrow \mathbb{R}$ are positive definite, $\mathcal{R} \in \mathbb{S}^n$ ($\mathcal{R}\succ  0$) is constant, $\mathcal{Q}:\mathcal{X} \rightarrow \mathbb{R}$ and $u:\mathcal{X} \rightarrow \mathbb{R}^m$. If $\Delta_{cl} (x)\succeq 0$ inside the domain $\mathcal{X}$, where
		{\small
			\begin{align}\label{eq:c5_32}
				\Delta_{cl} (x):= [f(x)+g(x)u(x)]^\top \mathcal{R} [f(x)+g(x)u(x)] 
				- \mathcal{Q}(x),
		\end{align}}then the state feedback $u(x)$ asymptotically stabilizes system (\ref{eq:c5_1}) around the origin. If, in addition, $\mathcal{E}(V,\rho) \subseteq \mathcal{X}$ for some $\rho \in\mathbb{R}_{>0}$, then $\mathcal{E}(V,\rho)$ is contractive invariant and it is an estimate of the closed-loop DOA. If $\mathcal{X}=\mathbb{R}^n$ and $V(x)$ is radially unbounded, then stability holds globally.
	\end{theorem}
	\begin{proof} Suppose that condition (\ref{eq:ox_4}) is feasible $\forall x \in \mathcal{X}$, subject to $\Delta_{cl} (x) \succeq 0$ at least in the same region. Then, by considering 
		\begin{align}\label{eq:ox_5}
			u_e = u_e (x) = f(x)+g(x)u(x),
		\end{align}	
		we obtain that 
		\begin{align}\label{eq:ox_35}
			\nabla V(x)^\top [f(x)+g(x)u(x)] \preceq -T(x) - \Delta_{cl} (x) \prec 0,
		\end{align}
		which implies that $\dot{V}(x)\prec 0$ in $\mathcal{X}$, i.e., the regional asymptotic stability of the closed-loop system around the origin. If $\mathcal{E}(V,\rho) \subseteq \mathcal{X}$ for some $\rho \in\mathbb{R}_{>0}$, then it follows that $\mathcal{E}(V,\rho)$ is contractive invariant, that is, $\mathcal{X}_c=\mathcal{E}(V,\rho)$, and it provides an estimate of the closed-loop DOA. The related conclusions about global stability are based on standard results from stability theory \cite{khali1}. 
	\end{proof}

	An interesting interpretation of these results is that $u(x)$ stabilizes all systems to which (\ref{eq:ox_4}) holds subject to $\Delta_{cl} (x)\succeq 0$, and this may be true for many different descriptions of $f(x)$ and $g(x)$. This work is about proving that such a \emph{robust stability framework} is particularly suitable for data-based control of polynomial systems. 
	
	Firstly, notice that relation (\ref{eq:ox_4}) in Thm. \ref{thm:thm_3} is equivalent to
	
	{\small
		\begin{align}\label{eq:ox_6}
			\begin{bmatrix} 1 \\ u_e \end{bmatrix}^\top	
			\begin{bmatrix} \mathcal{Q}(x)-T(x) & \star \\
				\Big( -\mathcal{R}[f(x)+g(x)u(x)]-\frac{1}{2}\nabla V(x) \Big) & \mathcal{R} \end{bmatrix} \begin{bmatrix} 1 \\ u_e \end{bmatrix} \succeq 0.	
	\end{align}}

	As a result, the following matrix inequality is a sufficient condition for the validity of (\ref{eq:ox_4}), if it holds $\forall x \in \mathcal{X}$,
	
	{\small
		\begin{align}\label{eq:ox_7}
			\mathcal{M}(x) = \begin{bmatrix} \mathcal{Q}(x)-T(x) & \star \\
				\Big( -\mathcal{R}[f(x)+g(x)u(x)]-\frac{1}{2}\nabla V(x) \Big) & \mathcal{R} \end{bmatrix} \succeq 0.	
	\end{align}}
	
	Unlike (\ref{eq:ox_6}), which is a function of $(x,u_e)$, condition (\ref{eq:ox_7}) depends only on the state $x$, and this is useful when applying Thm. \ref{thm:thm_3} for controller design using SOS programming. 
	
	Due to the product between $\mathcal{R}$ and $u(x)$, conditions (\ref{eq:ox_4}) and (\ref{eq:ox_7}) are \emph{nonlinear} on their decision variables. The same conclusion holds for function $\Delta_{cl} (x)$ defined in (\ref{eq:c5_32}). However, if we consider that 
	\begin{align}\label{eq:ox_37}
		\mathcal{R}=\gamma I_n,~~\gamma \in \mathbb{R}_{>0},
	\end{align}
	then relation (\ref{eq:ox_4}) can be presented as 
	\begin{align}\label{eq:ox_83}
		\nabla V(x)^\top u_e + T(x) \le \mathcal{Q}(x) ~~~~~~~~~~~~~~~~~~~~~~~~~~~\nonumber\\
		- 2[\gamma f(x)+g(x)u_\gamma (x)]^\top u_e + \gamma u_e^\top u_e,
	\end{align}	
	where $u_{\gamma}(x)$ is a new decision variable from which the value of $u(x)$ can be recovered, since
	\begin{align}\label{eq:ox_84}
		u_{\gamma}(x) = \gamma u(x).
	\end{align}	
	By the same token, condition (\ref{eq:ox_7}), when subject to (\ref{eq:ox_37}) and (\ref{eq:ox_84}), becomes the matrix condition below, which is linear on $\{  \mathcal{Q}(x),T(x),V(x),u_{\gamma} (x),\gamma \}$
	\begin{align}\label{eq:ox_127}
		\begin{bmatrix} \mathcal{Q}(x)-T(x) & \star \\
			\Big( -[\gamma f(x)+g(x)u_{\gamma}(x)]-\frac{1}{2}\nabla V(x) \Big) & \gamma I_n \end{bmatrix} \succeq 0.	
	\end{align}	
	
	Then, take into account that (\ref{eq:c5_32}) can be rearranged as follows
	\begin{align}
		\Delta_{cl} (x)= \varrho^\top (x) \gamma^{-1} \varrho (x)- \mathcal{Q}(x), \label{eq:ox_38} \\
		\varrho (x)= \gamma f(x)+g(x)u_{\gamma } (x).~~~~ \label{eq:ox_125}
	\end{align}
	Due to the nonlinearity of (\ref{eq:ox_38}), its application for controller design demands, in general, iterative procedures. When relations (\ref{eq:ox_83})-(\ref{eq:ox_125}) are employed in an algorithm, one needs to guarantee that between iterations it holds that $\Delta_{cl,k+1} (x) \succeq \Delta_{cl,k} (x)$, where from (\ref{eq:ox_38})
	\begin{align}
		\Delta_{cl,k} (x) := \varrho^\top_k (x) \gamma^{-1}_k \varrho_k (x)
		- \mathcal{Q}_k(x),\label{eq:ox_25}\\
		\varrho_k (x) := \gamma_k f(x)+g(x)u_{\gamma ,k} (x).~~~ \label{eq:ox_39}
	\end{align}  
	$\forall k\in\{0,1,2,\cdots \}$. If one obtains $\Delta_{cl,k} (x) \succeq 0$ after some $k=\kappa$ iterations, then closed-loop stabilization follows with $u_k (x) = u_{\gamma ,k} \slash \gamma_k$. In the lemma below, a set of conditions whose feasibility imply that $\Delta_{cl,k} (x)$ is non-decreasing are provided; the proof is similar to the one available in \cite[Prop. 2]{made12}.

	\begin{lemma}\label{lemma:lem_1} Let $\Delta_{cl,k} (x) $ be defined as in (\ref{eq:ox_25})-(\ref{eq:ox_39}), $\forall x \in \mathcal{X}$, $\forall k\in \{0,1,2,\cdots \}$, with $\mathcal{Q}_k : \mathcal{X} \rightarrow \mathbb{R}$, $u_{\gamma ,k} : \mathcal{X} \rightarrow \mathbb{R}^m$, $\gamma_k \in \mathbb{R}_{>0}$. If
		\begin{align}\label{eq:ox_130}
			\gamma_k - \gamma_{k+1} \succeq 0
		\end{align}
		and 
		\begin{align}\label{eq:ox_131}
			\varrho_{k+1} (x)^\top \gamma^{-1}_{k} \varrho_k (x) + \varrho_k (x)^\top \gamma^{-1}_{k} \varrho_{k+1} (x) \nonumber\\
			- 2 \varrho_k (x)^\top \gamma^{-1}_{k} \varrho_k (x)
			+ \mathcal{Q}_k (x) - \mathcal{Q}_{k+1} (x) \succeq 0,
		\end{align}
		then it follows that for all $k\in \{0,1,2,\cdots \}$, $\forall x \in \mathcal{X}$,
		\begin{equation}\label{eq:ox_132}
			\Delta_{cl,k+1} (x) \succeq \Delta_{cl,k} (x) .
		\end{equation}
	\end{lemma}
	\begin{proof} From (\ref{eq:ox_25}), relation (\ref{eq:ox_132}) is equivalent to
		\begin{align}\label{eq:ox_26}
			\varrho_{k+1} (x)^\top \gamma^{-1}_{k+1} \varrho_{k+1} (x) - \mathcal{Q}_{k+1} (x) \succeq \nonumber\\
			\varrho_{k} (x)^\top \gamma^{-1}_{k}  \varrho_{k} (x) - \mathcal{Q}_k (x).
		\end{align}
		In order to verify (\ref{eq:ox_26}), let us firstly suppose that (\ref{eq:ox_130}) holds. Then, when subject to (\ref{eq:ox_130}), condition (\ref{eq:ox_26}) holds true if 
		\begin{align}\label{eq:ox_27}
			\varrho_{k+1} (x)^\top \gamma^{-1}_{k} \varrho_{k+1} (x) - \mathcal{Q}_{k+1} (x) \succeq \nonumber\\ 
			\varrho_{k} (x)^\top \gamma^{-1}_{k}  \varrho_{k} (x)  
			- \mathcal{Q}_k (x).    
		\end{align}
		Next, consider the fact that 
		\begin{align}
			[\varrho_{k+1} (x) - \varrho_{k} (x)]^\top \gamma^{-1}_{k} [ \varrho_{k+1} (x) -\varrho_{k} (x)] \succeq 0,
		\end{align}
		which is equivalent to
		\begin{align}\label{eq:ox_28}
			\varrho_{k+1} (x)^\top \gamma^{-1}_{k} \varrho_{k+1} (x) \succeq 
			\varrho_{k+1} (x)^\top \gamma^{-1}_{k} \varrho_{k} (x) \nonumber\\
			+ \varrho_{k} (x)^\top \gamma^{-1}_{k} \varrho_{k+1} (x)
			- \varrho_{k} (x)^\top \gamma^{-1}_{k} \varrho_{k} (x).
		\end{align}
		Since the product $\varrho_{k+1} (x)^\top \gamma^{-1}_{k} \varrho_{k+1} (x)$ in the left-hand side of (\ref{eq:ox_27}) has a lower bound given by the expression on the right-hand side of (\ref{eq:ox_28}), one concludes that the inequality below is a sufficient condition for (\ref{eq:ox_27}) 
		\begin{align}\label{eq:ox_29}
			\varrho_{k+1} (x)^\top \gamma^{-1}_{k} \varrho_{k} (x)
			+ \varrho_{k} (x)^\top \gamma^{-1}_{k} \varrho_{k+1} (x) \nonumber\\
			- 2 \varrho_{k} (x)^\top \gamma^{-1}_{k} \varrho_{k} (x) \succeq \mathcal{Q}_{k+1} (x) - \mathcal{Q}_k (x),
		\end{align}
		which is identical to (\ref{eq:ox_131}). Thus, if (\ref{eq:ox_130})-(\ref{eq:ox_131}) hold for all $k\in \{0,1,2,\cdots \}$, $\forall x \in \mathcal{X}$, then (\ref{eq:ox_27}) and (\ref{eq:ox_132}) are verified as well.
	\end{proof}
	
	In the remainder of this work, we apply Thm. \ref{thm:thm_3} and different versions of Lem. \ref{lemma:lem_1} to derive new SOS strategies for data-driven control of polynomial systems using noisy measurements.

	\section{Applications to Data-Driven Control}
	
	This section presents two new iterative strategies for data-driven controller design. Strategy 1 uses the collected data to estimate a nominal model for the plant and, subsequently, employs a new definition of the set $\mathcal{C}$ and the results of Thm. \ref{thm:thm_3} to guarantee that all models in a large enough set  containing the reference dynamics and the whole set $\mathcal{C}$ are stable by a certain control law. In Strategy 2, instead of an indirect robust condition around a nominal model, a direct data-based approach is employed for controller design, where the data enters the stability conditions via the S-procedure.
	
	\subsection*{Strategy 1: Robust Stability Around a Reference Dynamics}
	
	If $\textbf{A}\succ 0$ and $\textbf{Q}\succ 0$ in condition (\ref{eq:ox_78}), then the set of matrices $[A~B]$ consistent with the data can be equivalently presented as
	
	{\small
	\begin{align}
	 \label{eq:ox_63}
			\mathcal{C}:= \{ [A~B]=Z^\top_{AB}: (Z_{AB} - \xi ) \textbf{Q}^{-1} ( Z_{AB} - \xi )^\top \preceq \textbf{A}^{-1} \}.
	\end{align}
	}
	
	Such an equivalence follows from the results of \cite[Thm. 1]{bekker}. Note that although $\textbf{Q}$ is not required to be invertible when employed in condition (\ref{eq:ox_78}), there always exists some matrix $\textbf{Q}_2 \succ \textbf{Q}$, $\textbf{Q}_2 \succ 0$, such that it holds with $\textbf{Q}$ replaced by $\textbf{Q}_2$ and the equivalent condition in (\ref{eq:ox_63}) is well defined. Thus, without loss of generality, we consider from this point on that $\textbf{Q}\succ 0$, which guarantees the existence of $\textbf{Q}^{-1}$. The fact that $\bf{A}$ is positive definite has been proven in \cite[Prop. 1]{biso1}. Then, let us consider the assumption below
	
	\begin{assumption}\label{ass:a_3} Let $\bf{A}\succ 0$ and $\textbf{Q}\succ 0$.
	\end{assumption} 
	    
	In this section, we prove that description (\ref{eq:ox_63}) can be employed in combination with Thm. \ref{thm:thm_3} for data-driven controller design. Note that according to Thm. \ref{thm:thm_3}, the quadratic function on the right-hand side of (\ref{eq:ox_4}) is non-positive if $\Delta_{cl} (x) \succeq 0$ and $u_e$ is given, for example, by (\ref{eq:ox_5}). However, there may exist a whole set of closed-loop dynamics, i.e., a set of descriptions of $u_e = u_e (x)$ around a \emph{reference dynamics} $u^\star_{e} (x) = f(x)+g(x)u(x)$, that verify the inequality below
	\begin{align}\label{eq:ox_65}
		\mathcal{Q}(x) - 2[f(x)+g(x)u(x)]^\top \mathcal{R} u_e + u_e^\top \mathcal{R} u_e \preceq 0,
	\end{align}	
	which can be equivalently presented as 
	\begin{align}\label{eq:ox_66}
		(u_e - u^\star_{e} (x))^\top \mathcal{R} (u_e - u^\star_{e} (x)) \preceq \Delta_{cl} (x),
	\end{align}	 
	where $\Delta_{cl} (x) $ is from (\ref{eq:c5_32})
	\begin{align}\label{eq:ox_99}
		\Delta_{cl} (x) = u^\star_{e}(x)^\top \mathcal{R} u^\star_{e} (x) - \mathcal{Q}(x).
	\end{align}	
	Thus, if (\ref{eq:ox_4}) or (\ref{eq:ox_7}) hold with $\Delta_{cl} (x) \succeq 0$, it follows that $V(x)$ is a CLF not only for $u^\star_{e} (x)$, but for all dynamics $u_e (x)$ that fulfill (\ref{eq:ox_66}). 
	
	In the following discussion consider that, from data, we obtain
	\begin{align}\label{eq:ox_90}
		[A_\star~B_\star] = \bf{\xi}^\top = [-\textbf{A}^{-1}\textbf{B}]^\top,
	\end{align}
	and that $u^\star_{e} (x)$, which lies at the center of an \emph{ellipsoidal set of models} that fulfill (\ref{eq:ox_66}), is given by 
	\begin{align}\label{eq:ox_87}
		u^\star_{e} (x) = [A_\star~B_\star] \begin{bmatrix} Z(x)\\ W(x) u(x)\end{bmatrix},
	\end{align}
	if some state feedback $u(x)$ is known. This means that   
	\begin{align}\label{eq:ox_88}
		f(x) = A_\star Z(x),~~g(x) = B_\star W(x),
	\end{align}
	and that if (\ref{eq:ox_4}) holds with $\Delta_{cl} (x) \succeq 0$, then there may exist a set of values for $[A~B]$ around $[A_\star~B_\star]$ such that (\ref{eq:ox_66}) is fulfilled with
	\begin{align}\label{eq:ox_89}
		u_e = u_e (x) = [A~B] \begin{bmatrix} Z(x)\\ W(x) u(x)\end{bmatrix}.
	\end{align} 
	
	The next lemma characterizes a set $\mathcal{C}_d$ of matrix coefficients $[A~B]$ that, if applied to (\ref{eq:ox_89}), lead to the feasibility of (\ref{eq:ox_65}) and (\ref{eq:ox_66}) around a nominal model $u^\star_{e} (x)$ given by (\ref{eq:ox_87}) and (\ref{eq:ox_88}). In Lemma \ref{lemma:lem_6}, $\mathcal{C}_d$ is not yet related to the size of the set $\mathcal{C}$ of models consistent with the data.

	\begin{lemma}\label{lemma:lem_6} Suppose that Assumptions \ref{ass:a_1}, \ref{ass:a_2} and \ref{ass:a_3} hold. Consider that $f(x)$ and $g(x)$ are given according to (\ref{eq:ox_88}), and that $u_{e} (x)$ is described by (\ref{eq:ox_89}), for a given control law $u(x)$. In addition, suppose that $\mathcal{Q}(x)$ is decomposed as 
		\begin{equation}\label{eq:ox_74}
			\mathcal{Q}(x)= \begin{bmatrix} Z(x)\\ W(x) u(x)\end{bmatrix}^\top Q_3
			\begin{bmatrix} Z(x)\\ W(x) u(x)\end{bmatrix},
		\end{equation}
		where $Q_3 \in \mathbb{S}^{N+q}$ is constant. In addition, consider that 
		\begin{align}\label{eq:ox_120}
			[A_\star~B_\star]^\top \mathcal{R} [A_\star~B_\star] - Q_3 \succeq 0,
		\end{align} 
		and define the convenient set below
		\begin{align}\label{eq:ox_91}
			\mathcal{C}_d:= \Big\{ [A~B]: \Big([A~B]-[A_\star~B_\star] \Big)^\top \mathcal{R} \Big([A~B]-[A_\star~B_\star] \Big) \nonumber\\
			\preceq [A_\star~B_\star]^\top \mathcal{R} [A_\star~B_\star] - Q_3 \Big\}.
		\end{align}
		It follows that $\forall [A~B]\in \mathcal{C}_d$, condition (\ref{eq:ox_65}) hold.
	\end{lemma} 
	\begin{proof}
		By multiplying the matrix inequality in (\ref{eq:ox_91}) from the left by $[Z(x)^\top~u(x)^\top W(x)^\top]$ and from the right by $[Z(x)^\top~u(x)^\top W(x)^\top]^\top$, one obtains
		
		{\small 
			\begin{align}\label{eq:ox_73}
				\begin{bmatrix} Z(x)\\ W(x) u(x)\end{bmatrix}^\top \cdot \Big( [A~B] - [A_\star~B_\star]\Big)^\top \mathcal{R} \Big( [A~B] - [A_\star~B_\star]\Big) \cdot \nonumber\\
				\begin{bmatrix} Z(x)\\ W(x) u(x)\end{bmatrix} \preceq \begin{bmatrix} Z(x)\\ W(x) u(x)\end{bmatrix}^\top \cdot [A_\star~B_\star]^\top \mathcal{R} [A_\star~B_\star] \nonumber\\
				\cdot \begin{bmatrix} Z(x)\\ W(x) u(x)\end{bmatrix} - \mathcal{Q}(x).
		\end{align}}
	
		As a result, (\ref{eq:ox_66}) is feasible, and it implies (\ref{eq:ox_65}). 
	\end{proof}

	Next, we present sufficient conditions for the set $\mathcal{C}_d$ in (\ref{eq:ox_91}) to be at least as large as the set $\mathcal{C}$ of all plants that could have generated the data. In the lemma below, notice that condition (\ref{eq:ox_63}) is the same as
	\begin{align}\label{eq:ox_70}
		\mathcal{C}:= \{ [A~B]: \Big([A~B]-[A_\star~B_\star] \Big)^\top \cdot \textbf{Q}^{-1}  \nonumber\\
		\cdot \Big([A~B]-[A_\star~B_\star] \Big)\preceq \textbf{A}^{-1} \}.
	\end{align}	
	
	\begin{lemma}\label{lemma:lem_4} Let the sets $\mathcal{C}$ and $\mathcal{C}_d$ be defined, respectively, by (\ref{eq:ox_70}) and (\ref{eq:ox_91}). It follows that $\mathcal{C}_d \supseteq \mathcal{C}$, if
		\begin{equation}\label{eq:ox_85}
			\mathcal{R} \preceq \textbf{Q}^{-1}
		\end{equation}
		and
		\begin{equation}\label{eq:ox_86}
			[A_\star~B_\star]^\top \mathcal{R} [A_\star~B_\star] - Q_3 \succeq \textbf{A}^{-1} .
		\end{equation}
	\end{lemma} 
	\begin{proof} The matrix inequality in condition (\ref{eq:ox_70}) is equivalent to
		\begin{equation}
			\begin{bmatrix} \textbf{A}^{-1} & \star \\ \Big( [A~B]-\xi^\top \Big) & \textbf{Q}\end{bmatrix} \succeq 0, \nonumber
		\end{equation}
		and that of condition (\ref{eq:ox_91}) is equivalent to
		\begin{equation}
			\begin{bmatrix} \xi \mathcal{R} \xi^\top - Q_3 & \star \\ \Big( [A~B]-\xi^\top \Big) & \mathcal{R}^{-1}\end{bmatrix} \succeq 0. \nonumber
		\end{equation}
		Thus, if (\ref{eq:ox_85})-(\ref{eq:ox_86}) are fulfilled, the following inequality holds
		{\small
			\begin{equation}\label{eq:ox_121}
				\begin{bmatrix} \xi \mathcal{R} \xi^\top - Q_3 & \star \\ \Big( [A~B]-\xi^\top \Big) & \mathcal{R}^{-1}\end{bmatrix} - \begin{bmatrix} \textbf{A}^{-1} & \star \\ \Big( [A~B]-\xi^\top \Big) & \textbf{Q}\end{bmatrix} \succeq 0, \nonumber
		\end{equation}} 
		which means that $\forall [A~B]\in \mathcal{C}$, one also has $[A~B]\in \mathcal{C}_d$, i.e., $\mathcal{C}_d \supseteq \mathcal{C}$.
	\end{proof}
	
	This lemma is key to allow connecting the stabilizability results of Thm. \ref{thm:thm_3} with the well-known definition in (\ref{eq:ox_78}) of a set $\mathcal{C}$ of dynamics consistent with data. Such a relation is addressed in the next theorem, where it is proved that if reference model (\ref{eq:ox_88}) obtained through data is stabilized by some state feedback $u(x)$ and one can prove that a large enough set $\mathcal{C}_d \supseteq \mathcal{C}$ of other models is also held stable by the same feedback control, then the asymptotic stabilization of all plants consistent with the data may follow. 
	
	\begin{theorem}\label{thm:thm_5} Consider a nonlinear system that verifies Assumption \ref{ass:a_1} and is given by (\ref{eq:sos_2}), where $[A~B]$ is unknown. Suppose that a collected data set verifies Assumptions \ref{ass:a_2} and \ref{ass:a_3}, where matrices (\textbf{A},\textbf{B},\textbf{C}) are determined according to (\ref{eq:ox_18}). Assume that for some state feedback $u(x)$, a reference closed-loop dynamics given by (\ref{eq:ox_90})-(\ref{eq:ox_88}) fulfills condition (\ref{eq:ox_4}) $\forall (x,u_e )\in \mathcal{X}\times \mathbb{R}^n$, with $\mathcal{Q}(x)$ given by (\ref{eq:ox_74}). If (\ref{eq:ox_85})-(\ref{eq:ox_86}) hold true, then $u(x)$ asymptotically stabilizes all plants (\ref{eq:sos_2}) where $[A~B]\in \mathcal{C}$, $\forall x(0) \in \mathcal{X}$. If $\mathcal{E}(V,\rho) \subseteq \mathcal{X}$ for some $\rho \in\mathbb{R}_{>0}$, then $\mathcal{E}(V,\rho)$ is contractive invariant and it is an estimate of the closed-loop DOA. If $\mathcal{X}=\mathbb{R}^n$ and $V(x)$ is radially unbounded, then stability holds globally.
	\end{theorem}
	\begin{proof} If, for some $u(x)$, condition (\ref{eq:ox_4}) is feasible $\forall (x,u_e )\in \mathcal{X}\times \mathbb{R}^n$, with $f(x)$ and $g(x)$ given by (\ref{eq:ox_88}), $\mathcal{Q}(x)$ as in (\ref{eq:ox_74}), and inequality (\ref{eq:ox_86}) holds, then $\Delta_{cl} (x) \succeq 0$ in (\ref{eq:ox_99}) and for all $u_e$ given by (\ref{eq:ox_89}), where $[A~B]\in \mathcal{C}_d$, conditions (\ref{eq:ox_73}) and (\ref{eq:ox_65}) are verified as well. Notice that (\ref{eq:ox_86}) implies (\ref{eq:ox_120}), and thus $\mathcal{C}_d$ is well-defined. If (\ref{eq:ox_85}) is feasible, then, according to Lem. \ref{lemma:lem_4}, $\mathcal{C}_d \supseteq \mathcal{C}$, and $u(x)$ locally asymptotically stabilizes all closed-loop dynamics (\ref{eq:ox_89}) where $[A~B]\in \mathcal{C}$. The set $\mathcal{C}_d$ leads to a special case of dynamics that are stabilizable by $u(x)$, and the set $\mathcal{C}$, which in principle simply describes a set of dynamics consistent with the data, happens to be a subset of $\mathcal{C}_d$, if (\ref{eq:ox_85})-(\ref{eq:ox_86}) are fulfilled. The conclusions on an estimated DOA and on global stability are standard.
	\end{proof} 
	   
	By employing Thm. \ref{thm:thm_5} and an adaptation of Lem. \ref{lemma:lem_1}, a constructive procedure for data-driven feedback design is proposed. We consider some  $L(x) \in \mathcal{P}[x]$, $\mathcal{Q}(x) \in \mathcal{P}[x]$, $u(x) \in \mathcal{P}^{m\times 1} [x]$, $V(x) \in \mathcal{P}[x]$, $T(x) \in \mathcal{P}[x]$, $\alpha(x) \in \Sigma^{n+1} [x]$, and some $\mathcal{R}\in\mathbb{S}$, $\mathcal{R}\succ 0$. Then, using reference model (\ref{eq:ox_88}), we employ the following condition 
	
	{\small
		\begin{align}\label{eq:ox_133}
			\begin{bmatrix} \mathcal{Q}(x)-T(x) & \star \\
				\Big( -\mathcal{R}[f(x)+g(x)u(x)]-\frac{1}{2}\nabla V(x) \Big) & \mathcal{R} \end{bmatrix} \nonumber\\
			-\alpha (x)(1-L(x)) \in \Sigma^{n+1} [x],	
	\end{align}}
	which is an SOS-matrix certificate that is sufficient for the feasibility of (\ref{eq:ox_7}), $\forall x \in \mathcal{E}(L,1)$. Notice that this sufficiency result holds because the feasibility of SOS condition (\ref{eq:ox_133}) is equivalent to the existence of a set of polynomial matrices $W_j(x)$ of compatible dimensions, $j\in\{1,2,\cdots,n_w\}$, $n_w \in \mathbb{N}_{>0}$, such that
	
	{\small
		\begin{align}\label{eq:ox_136}
				\mathcal{M}(x) = \sum^{n_w}_{j=1} W_j(x) W_j(x)^\top + \alpha (x)(1-L(x)),
		\end{align}
	} 
	
	and since $1-L(x)\succeq 0$, $\forall x\in \mathcal{E}(L,1)$, and $\alpha (x)\succeq 0$ by assumption, it follows that $\mathcal{M}(x)\succeq 0$ as well, $\forall x\in \mathcal{E}(L,1)$.
	
	In addition, consider that
	\begin{align}
		L(x) - \beta_l ||x||^{2 n_l} \in \Sigma [x],~~~~~~~~~~~~~ \label{eq:acc_23}\\
		V (x) - \beta_v ||x||^{2 n_v} - s_{vl}(x) (1-L(x)) \in \Sigma [x] ,\label{eq:acc_21}\\
		T (x) - \beta_t ||x||^{2 n_t} - s_{tl}(x) (1-L(x)) \in \Sigma [x],\label{eq:acc_22}
	\end{align}
	for some $\{ \beta_v ,\beta_t ,\beta_l \}\in \mathbb{R}_{>0}$, $\{ n_v ,n_t ,n_l \}\in \mathbb{N}_{>0}$, $\{s_{tl},s_{vl}\} \in \Sigma [x]$. Conditions (\ref{eq:acc_23})-(\ref{eq:acc_22}) mean that $\mathcal{X} = \mathcal{E}(L,1)$, and that $V(x)$ and $T(x)$ need to be positive definite only in that region.
	Next, let
	\begin{align}\label{eq:ox_115}
		\mathcal{R} = \textbf{Q}^{-1},
	\end{align}
	which fulfills (\ref{eq:ox_85}). Then, (\ref{eq:ox_133}) is tested iteratively and, in order to guarantee that $\Delta_{cl,k+1} (x) \succeq \Delta_{cl,k} (x)$ between repetitions, the conditions of the following lemma are taken into account, where for $\forall k\in \{0,1,2,\cdots \}$, $\forall x\in\mathcal{X}$, we  define 
	\begin{align}\label{eq:ox_92}
		\bar{\varrho}_k (x)= f(x)+g(x)u_{k} (x).
	\end{align}

	\begin{lemma}\label{lemma:lem_2} Suppose that $\mathcal{R} \succ 0$ is known. In this case, $\Delta_{cl,k} (x) $ is defined as follows, $\forall k\in \{0,1,2,\cdots \}$, $\forall x\in\mathcal{X}$, 
		\begin{align}\label{eq:ox_93}
			\Delta_{cl,k} (x) := \bar{\varrho}^\top_k (x) \mathcal{R} \bar{\varrho} (x) 
			- \mathcal{Q}_k(x),
		\end{align}
		with $\mathcal{Q}_k (x) \in \mathcal{P}[x]$, $u_k (x) \in \mathcal{P}^{m\times 1} [x]$. If $\forall k\in \{0,1,2,\cdots \}$
		\begin{align}\label{eq:ox_94}
			\bar{\varrho}_{k+1} (x)^\top \mathcal{R} \bar{\varrho}_k (x) + \bar{\varrho}_k (x)^\top \mathcal{R} \bar{\varrho}_{k+1} (x) \nonumber\\
			- 2 \bar{\varrho}_k (x)^\top \mathcal{R} \bar{\varrho}_k (x)
			+ \mathcal{Q}_k (x) - \mathcal{Q}_{k+1} (x) \in \Sigma [x],
		\end{align}
		then it follows that $\forall x\in\mathcal{X}$
		\begin{align}\label{eq:ox_95}
			\Delta_{cl,k+1} (x) \succeq \Delta_{cl,k} (x).
		\end{align}
	\end{lemma}
	\begin{proof} The proof is similar to that of Lem. \ref{lemma:lem_1} and is omitted.
	\end{proof}
	         
	In addition, we employ the condition below, subject to (\ref{eq:ox_86}), as \emph{a stopping criterion}
	 
	{\small
		\begin{align}\label{eq:ox_114}
			\begin{bmatrix} \begin{bmatrix} Z(x)\\ W(x) u(x)\end{bmatrix}^\top Q_3
				\begin{bmatrix} Z(x)\\ W(x) u(x)\end{bmatrix}-T(x) & \star \\
				\Big( -\mathcal{R}[f(x)+g(x)u(x)]-\frac{1}{2}\nabla V(x) \Big) & \mathcal{R} \end{bmatrix} \nonumber\\
			- \alpha(x) (1-L(x)) \in \Sigma^{n+1} [x],
	\end{align}}

	$Q_3 \in \mathbb{S}^{N+q}$. According to Thm. \ref{thm:thm_5}, if (\ref{eq:ox_114}) is satisfied subject to (\ref{eq:acc_23})-(\ref{eq:acc_22}), then our problem has been solved $\forall x(0) \in \mathcal{E}(L,1)$. 
	Once the controlled system is held stable in some $\mathcal{E}(L,1)$, a contractive invariant set can be obtained by determining the largest set $\mathcal{E}(V,\rho)$ that fulfills $\mathcal{E}(V,\rho) \subseteq \mathcal{E}(L,1)$, with $\rho \in \mathbb{R}_{>0}$. This can be accomplished by solving the SOS condition below for some $\rho$ and some $\psi (x) \in \Sigma [x]$, for given $L(x)$ and $V(x)$ that satisfy (\ref{eq:acc_23})-(\ref{eq:acc_21}), 
	\begin{align}\label{eq:acc_20}
		(1-L) - \psi (\rho - V) \in \Sigma [x].
	\end{align} 
	$\overline{\underline{\textbf{Algorithm 1:~~~~~~~~~~~~~~~~~~~~~~~~~~~~~~~~~~~~~~~~~~~~~~~~~~~~}}}$
	\textbf{Inputs:} \textbf{A}, \textbf{Q}, $\xi$, $Z(x)$, $W(x)$. Consider that $f(x)$ and $g(x)$ are given by (\ref{eq:ox_88}). Set $k_{max} \in\mathbb{N}_{>0}$, $u=\emptyset$, $\mathcal{R} = \textbf{Q}^{-1}$, $\left\{ \beta_v ,\beta_t ,\beta_l \right\} \in \mathbb{R}_{>0}$, $\left\{ n_v ,n_t ,n_l  \right\}\in \mathbb{N}_{>0}$. Set the degrees on $x$ of all decision variables, $\forall k\in \{0,1,2,\cdots ,k_{max} \}$.\\$ $\\
	STEP 1. Set $k=0$ and initialize $\alpha_0\in\Sigma^{n+1} [x]$, $s_{vl,0}\in \Sigma [x]$ and $s_{tl,0} \in \Sigma [x]$. Find $(V_{0}, T_{0}, L)$, subject to (\ref{eq:acc_23})-(\ref{eq:acc_22}), and $(\mathcal{Q}_0 ,u_0)$, such that (\ref{eq:ox_133}) holds.
	$ $\\
	\textbf{If} STEP 1 is feasible, then find $Q_3 \in \mathbb{S}^{N+m}$, $\{V, T, s_{vl}, s_{tl}\}$ subject to (\ref{eq:acc_21})-(\ref{eq:acc_22}), and $\alpha$ such that (\ref{eq:ox_114}) holds with $u=u_0$ and subject to (\ref{eq:ox_86}).\\
	$\cdot$ \textbf{If} the conditions above are feasible, then $V=V_0$, $u = u_0$. Go to STEP 3.\\
	$\cdot$ \textbf{Else} go to STEP 2.\\
	$\cdot$ \textbf{End-if}.\\
	\textbf{Else} \textbf{STOP}.\\
	\textbf{End-if}.\\$ $\\ 
	STEP 2. \textbf{While} $0\le k < k_{max}$:\\ 
	Find $(V_{k+1} ,T_{k+1}, s_{vl,k+1}, s_{tl,k+1})$ subject to (\ref{eq:acc_21})-(\ref{eq:acc_22}), some $\alpha_{k+1}$, and $(\mathcal{Q}_{k+1} ,u_{k+1})$ such that (\ref{eq:ox_133}) holds subject to (\ref{eq:ox_94}).\\
	\textbf{If} feasible, find $Q_3 \in \mathbb{S}^{N+m}$, $\{V, T, s_{vl}, s_{tl}\}$ subject to (\ref{eq:acc_21})-(\ref{eq:acc_22}), and $\alpha$ such that (\ref{eq:ox_114}) holds with $u=u_{k+1}$ and subject to (\ref{eq:ox_86}).\\
	$\cdot$ \textbf{If} this is feasible, then $V=V_{k+1}$, $u = u_{k+1}$. Go to STEP 3.\\
	$\cdot$ \textbf{End-if}.\\
	\textbf{Else} \textbf{STOP}.\\
	\textbf{End-if}.\\
	$k=k+1$.\\
	\textbf{End-while}. STOP. \\$ $\\
	STEP 3. Set some $\rho \in\mathbb{R}_{>0}$, determine $\psi \in \Sigma [x]$ that solves (\ref{eq:acc_20}). Consider higher values of $\rho$ and solve (\ref{eq:acc_20}) again, until it is no longer feasible, then STOP.\\$ $\\
	\textbf{Ouputs:} $u(x)$, $\mathcal{E}(L,1)$, $\mathcal{E}(V,\rho)$.\\
	$\overline{~~~~~~~~~~~~~~~~~~~~~~~~~~~~~~~~~~~~~~~~~~~~~~~~~~~~~~~~~~~~~~~~~~~~~~~~}$
	$ $\\
	In STEP 1, $\alpha_0$ is set by the designer. Then, the algorithm determines $V_0$, $u_0$, $L$ and all other decision variables. If stopping criterion (\ref{eq:ox_114}) is not verified, then we move to STEP 2, the iterative part of the method, where condition (\ref{eq:ox_94}) is imposed to guarantee that $\Delta_{cl,k+1} \succeq \Delta_{cl,k}$. Notice, however, that $L$ is not updated.

	\subsection*{Strategy 2: Direct Data-Driven Control}
		
	In this alternative approach, as said before, a direct data-based approach is employed for controller design, using S-procedure. 
	
	\begin{proposition}\label{prop:prp_1} Consider polynomial system (\ref{eq:sos_2}) and Assumption \ref{ass:a_2}. Suppose that there exist positive definite $V(x)\in \Sigma [x]$, $T(x) \in \Sigma [x]$ and $L(x)\in \Sigma [x]$; $\sigma (x) \in \Sigma [x]$, $\beta (x) \in \Sigma [x]$, $\alpha (x) \in \Sigma^{1+n+q+N} [x]$, $\mathcal{Q}(x)\in\mathcal{P} [x]$, $Q_e(x)\in\mathcal{P}^{n\times n} [x]$, $\gamma \in \mathbb{R}_{>0}$ and some $u_{\gamma}(x)\in\mathcal{P}^{m\times 1} [x]$, such that 
		
	{\small
			\begin{align}\label{eq:ox_8}
				\begin{bmatrix}
					\mathcal{Q}(x)-T(x)  & \star & \star\\
					-\frac{1}{2}\nabla V(x) & \gamma I_n + \sigma (x) \textbf{C} & \star\\
					-\begin{bmatrix} \gamma Z(x)\\ W(x)u_{\gamma}(x) \end{bmatrix} & 
					\sigma (x) \textbf{B} & \sigma (x) \textbf{A}
				\end{bmatrix} \nonumber\\
				- \alpha (x) (1 - L(x)) \in \Sigma^{1+n+q+N} [x],
	\end{align}}
	
	and
	{\small
			\begin{align}\label{eq:ox_10}
				\begin{bmatrix} 0\\ 
					\begin{bmatrix} \gamma Z(x)\\ W(x)u_{\gamma}(x) \end{bmatrix} 
				\end{bmatrix} \gamma^{-1} \begin{bmatrix}\star \end{bmatrix} -
				\begin{bmatrix} Q_e (x) - \beta (x) \textbf{C} & \star\\ 
					-\beta (x)\textbf{B} & -\beta (x) \textbf{A}
				\end{bmatrix} \nonumber\\
				= \bar{\Delta}_{cl} (x) \in \Sigma^{n+q+N} [x],
	\end{align}}

	with 
	\begin{align}\label{eq:ox_12}
			Q_e (x) = Q_e (x)^\top,~~tr\Big( Q_e (x) \Big) \succeq \mathcal{Q}(x). 
	\end{align}
	Then, the polynomial state feedback control law given by
	\begin{align}\label{eq:ox_11}
			u(x) = u_{\gamma}(x) / \gamma 
	\end{align}
	asymptotically stabilizes system (\ref{eq:sos_2}), $\forall [A~B]\in \mathcal{C}$, $\forall x(0) \in \mathcal{E}(L,1)$. If, in addition, $\mathcal{E}(V,\rho) \subseteq \mathcal{E}(L,1)$ for some $\rho \in\mathbb{R}_{>0}$, then $\mathcal{E}(V,\rho)$ is contractive invariant and it is an estimate of the closed-loop DOA. If $\mathcal{X}=\mathbb{R}^n$ and $V(x)$ is radially unbounded, then stability is global.
	\end{proposition}
	\begin{proof} Consider the following vector $E \in \mathbb{R}^{1+n+q+N}$
		
		{\small
			\begin{align}\label{eq:ox_15}
				E = \begin{bmatrix} 1 & u^\top_e & u^\top_e [A~B] \end{bmatrix}^\top.
			\end{align}}
		
		By multiplying matrix condition (\ref{eq:ox_8}) from the left by $E^\top$ and from the right by $E$, we obtain the following polynomial on $(x,u_e)$
		
		{\small
			\begin{align}
					\Bigg( \mathcal{Q}(x) - T(x) - 2\Bigg[ [A~B]\begin{bmatrix} \gamma Z(x)\\ W(x)u_\gamma (x) \end{bmatrix} + \frac{1}{2} \nabla V(x)\Bigg]^\top  u_e \nonumber\\
					+ u_e^\top \mathcal{R} u_e \Bigg) - \alpha (x)(1-L(x)) + \sigma(x) u^\top_e \Theta u_e \succeq 0, 	\nonumber
			\end{align}}
		
		due to (\ref{eq:ox_37}), where, from (\ref{eq:ox_17}), let us define
		\begin{align}\label{eq:ox_20}
				\Theta =  [I~A~B]\cdot \left[ \begin{array}{c|c}
					\textbf{C} & \textbf{B}^\top \\ \hline
					\textbf{B} & \textbf{A}
				\end{array} \right] \cdot [\star]^\top \preceq 0.	
		\end{align}
		This means that $\forall [A~B]\in \mathcal{C}$ condition (\ref{eq:ox_21}) holds $\forall x \in \mathcal{E}(L,1)$
		
		{\small
			\begin{align}\label{eq:ox_21}
					\nabla V(x)^\top u_e + T(x) \le \mathcal{Q}(x)~~~~~~~~~\nonumber\\
					- 2\Bigg[ [A~B]\begin{bmatrix} Z(x)\\ W(x)u(x) \end{bmatrix}\Bigg]^\top \mathcal{R} u_e + u_e^\top \mathcal{R} u_e .
			\end{align}}
			
		Notice that this does not guarantee closed-loop stabilization, since the non-negativity of $\Delta_{cl} (x)$ in (\ref{eq:c5_32}) still needs to be certified. In fact, $\Delta_{cl} (x) \succeq 0$ is verified if conditions (\ref{eq:ox_10})-(\ref{eq:ox_12}) are feasible. In order to come to this conclusion, note that by multiplying (\ref{eq:ox_10}) from the left by $[I~A~B]$ and from the right by $[I~A~B]^\top$ it follows that
		
		{\small
				\begin{align}\label{eq:ox_24}
					\begin{bmatrix} \gamma AZ(x) + BW(x)u_\gamma (x) \end{bmatrix} \gamma^{-1} \begin{bmatrix}\star  \end{bmatrix} - Q_e (x) + \beta (x) \Theta \succeq 0.
		\end{align}}
	
		Using (\ref{eq:ox_11}), as well as $\beta(x)\succeq 0$ and $\Theta \preceq 0$, condition (\ref{eq:ox_24}) leads to 
		
		{\small
				\begin{align}\label{eq:ox_22}
						\begin{bmatrix} AZ(x) + BW(x)u(x) \end{bmatrix} \gamma \begin{bmatrix}\star  \end{bmatrix} - Q_e (x) \succeq 0,
				\end{align}
		}
		
		for all $[A~B]\in \mathcal{C}$. Now, consider (\ref{eq:ox_12}) and the following trace properties, which hold for all square matrices $X$ and $Y$ and all $c\in \mathbb{R}$ 
			\begin{align}
					tr(X+Y)=tr(X)+tr(Y),~~~~~~~ \nonumber\\
					tr(XY)=tr(YX),~~tr(cX)=c\cdot tr(X).\nonumber
			\end{align}
		Due to the above relations, and also by considering $\mathcal{R}=\gamma I_n$ and (\ref{eq:ox_22}), we conclude that condition (\ref{eq:c5_32}) repeated below is valid
		
		{\small
			\begin{align}
					\Delta_{cl} (x) = \begin{bmatrix} AZ(x) + BW(x)u(x) \end{bmatrix}^\top \mathcal{R} \begin{bmatrix}\star  \end{bmatrix} - \mathcal{Q}(x) \succeq 0, \nonumber
			\end{align}
		}
		
		and, from Thm. \ref{thm:thm_3}, closed-loop stabilization by the control law in (\ref{eq:ox_11}) is guaranteed $\forall [A~B]\in \mathcal{C}$.
		\end{proof}
	
		Conditions (\ref{eq:ox_8}) and (\ref{eq:ox_12}) are \emph{linear} on all decision variables, whereas condition (\ref{eq:ox_10}) is \emph{nonlinear} and must be solved through iterative procedures. An adaptation of the recurrent inequalities proposed in Lem. \ref{lemma:lem_1} must be developed.$ $\\
		
		\begin{lemma}\label{lemma:lem_5} Let us define $\forall k\in \{0,1,2,\cdots \}$ the following functions based on conditions (\ref{eq:ox_10})-(\ref{eq:ox_12})
			\begin{align}\label{eq:ox_31}
				\bar{S}_k (x) = \begin{bmatrix} 0\\ \begin{bmatrix} \gamma_k Z(x)\\ W(x)u_{\gamma ,k} (x) \end{bmatrix}  \end{bmatrix} 
			\end{align}
			and
			\begin{align}\label{eq:ox_32}
				\bar{Q}_k (x)= \begin{bmatrix} Q_{e,k} (x) - \beta_k (x) \textbf{C} & \star\\ 
					-\beta_k (x)\textbf{B} & -\beta_k (x) \textbf{A}
				\end{bmatrix}.
			\end{align} 
			In accordance with (\ref{eq:ox_10}), let $\bar{\Delta}_{cl,k} (x)$ be defined as:
			\begin{align}\label{eq:ox_34}
				\bar{\Delta}_{cl,k} (x) = \bar{S}_{k} (x) \gamma^{-1}_{k} \bar{S}^\top_k (x) - \bar{Q}_{k} (x).
			\end{align}
			$\forall x \in \mathcal{X}$, $\forall k\in \{0,1,2,\cdots \}$. If
			\begin{align}\label{eq:sos_45}  
				\gamma_k - \gamma_{k+1} \succeq 0
			\end{align}
			and 
			\begin{align}\label{eq:ox_34}
				\bar{S}_{k+1} \gamma^{-1}_{k} \bar{S}^\top_k + \bar{S}_k \gamma^{-1}_{k} \bar{S}^\top_{k+1} - 2 \bar{S}_k \gamma^{-1}_{k} \bar{S}^\top_k \nonumber\\
				+ \bar{Q}_k - \bar{Q}_{k+1} \in \Sigma^{n+q+N} [x], 
			\end{align}
			then it follows that for all $k\in \{0,1,2,\cdots \}$
			\begin{align}\label{eq:sos_47}
				\bar{\Delta}_{cl,k+1} (x) \succeq \bar{\Delta}_{cl,k} (x).
			\end{align}			
		\end{lemma}
		\begin{proof} A proof is similar to the one provided for Lem. \ref{lemma:lem_1} and will be omitted. 
		\end{proof}
		   
		If, for any $k\in \{0,1,2,\cdots \}$ one obtains $\bar{\Delta}_{cl,k} (x)\succeq 0$, then, from Thm. \ref{prop:prp_1}, we guarantee that $\Delta_{cl} (x) \succeq 0$ and feedback stabilization has been achieved $\forall x(0) \in \mathcal{X}$. Notice that we have $\bar{S}_k (\gamma_k ,u_{\gamma ,k} (x))$ and $\bar{Q}_k (Q_{e,k} (x), \beta_k (x))$, where the explicit dependence on all decision variables is made. 
		
		We have empirically verified that employing condition (\ref{eq:ox_10}) as a stopping criterion is not efficient, since a stabilizing controller may be obtained many iterations before $\bar{\Delta}_{cl} (x)\succeq 0$ is verified. Employing condition   (\ref{eq:ox_114}) as a stopping criterion instead, subject to (\ref{eq:ox_85})-(\ref{eq:ox_86}), proved more suitable in simulations. These constraints are less computationally demanding and less conservative than (\ref{eq:ox_10}).
		\\$ $\\
		$\overline{\underline{\textbf{Algorithm 2:~~~~~~~~~~~~~~~~~~~~~~~~~~~~~~~~~~~~~~~~~~~~~~~~~~~~}}}$
		\textbf{Inputs:} \textbf{A}, \textbf{B}, \textbf{C}, \textbf{Q}, $Z(x)$, $W(x)$. Set $k_{max} \in\mathbb{N}_{>0}$, $u=\emptyset$, $\left\{ \beta_v ,\beta_t ,\beta_l \right\} \in \mathbb{R}_{>0}$, $\left\{ n_v ,n_t ,n_l  \right\}\in \mathbb{N}_{>0}$. Set the degrees on $x$ of all decision variables, $\forall k\in \{0,1,2,\cdots ,k_{max} \}$.
		\\$ $\\
		STEP 1. Set $k=0$ and initialize $\alpha_0\in\Sigma^{1+n+q+N} [x]$, $s_{vl,0}\in \Sigma [x]$ and $s_{tl,0} \in \Sigma [x]$. Find $(V_{0}, T_{0}, L)$, subject to (\ref{eq:acc_23})-(\ref{eq:acc_22}), and $(\mathcal{Q}_0 ,\gamma_0, u_{\gamma,0}, \sigma_0)$, such that (\ref{eq:ox_8}) holds.
		$ $\\
		\textbf{If} STEP 1 is feasible, then find $Q_3 \in \mathbb{S}^{N+m}$, $\{V, T, s_{vl}, s_{tl}\}$ subject to (\ref{eq:acc_21})-(\ref{eq:acc_22}), and $\alpha$ such that (\ref{eq:ox_114}) holds with $u=u_0$ and subject to (\ref{eq:ox_86}) and $\mathcal{R}=\textbf{Q}^{-1}$.\\
		$\cdot$ \textbf{If} the above is feasible, then $V=V_0$, $u = u_0$. Go to STEP 3.\\
		$\cdot$ \textbf{Else} go to STEP 2.\\
		$\cdot$ \textbf{End-if}.\\
		\textbf{Else} \textbf{STOP}.\\
		\textbf{End-if}.\\$ $\\ 
		STEP 2. \textbf{While} $0\le k < k_{max}$:\\ 
		Find $(V_{k+1} ,T_{k+1}, s_{vl,k+1}, s_{tl,k+1})$ subject to (\ref{eq:acc_21})-(\ref{eq:acc_22}), some $\alpha_{k+1}$, and $(\mathcal{Q}_{k+1} , \gamma_{k+1}, u_{\gamma, k+1}, \sigma_{k+1} )$ such that (\ref{eq:ox_8}) holds subject to (\ref{eq:sos_45})-(\ref{eq:ox_34}).\\
		\textbf{If} feasible, then find $Q_3 \in \mathbb{S}^{N+m}$, $\{V, T, s_{vl}, s_{tl}\}$ subject to (\ref{eq:acc_21})-(\ref{eq:acc_22}), and $\alpha$ such that (\ref{eq:ox_114}) holds with $u=u_{k+1}$ and subject to (\ref{eq:ox_86}) and $\mathcal{R}=\textbf{Q}^{-1}$.\\
		$\cdot$ \textbf{If} this is feasible, then $V=V_{k+1}$, $u = u_{k+1}$. Go to STEP 3.\\
		$\cdot$ \textbf{End-if}.\\
		\textbf{Else} \textbf{STOP}.\\
		\textbf{End-if}.\\
		$k=k+1$.\\
		\textbf{End-while}. STOP. \\$ $\\
		STEP 3. The same as in Algorithm 1.\\$ $\\
		\textbf{Ouputs:} $u(x)$, $\mathcal{E}(L,1)$, $\mathcal{E}(V,\rho)$.\\
		$\overline{~~~~~~~~~~~~~~~~~~~~~~~~~~~~~~~~~~~~~~~~~~~~~~~~~~~~~~~~~~~~~~~~~~~~~~~~}$
		
		Experiments confirm that the output of STEP 1 is usually quite close to a true solution of our problem, and it provides a great initialization for STEP 2 of the algorithm. Unlike other SOS-strategies available for data-driven control, Algorithm 2 is quite robust with regards to the data themselves, for a given level of noise.

		\section{Numerical Results}
		
		In this section, simulations are performed using SOSTOOLS \cite{prajna2} and MOSEK ApS in MATLAB, where codes are available at \cite{DMgithub}. 
		Consider the model below from \cite{moez1}
		
		{\scriptsize
		\begin{align}
			f(x) = \begin{bmatrix}
				-x_1 +x_2 +x^2_1+x_1 x_2 -x^3_1 +x^2_1 x_2 -x_1 x^2_2 +2x^3_2 \\
				-x_1 +1.5x_2 -x^2_1 -0.5x_1 x_2 -x^3_1 -x^2_1 x_2 +0.5x_1 x^2_2 -2x^3_2
			\end{bmatrix},\nonumber
		\end{align}
		}
		\begin{equation}
			g(x) = \begin{bmatrix} 0 & 1\end{bmatrix}^\top ,\nonumber
		\end{equation}
		where $\delta_{fg}=3$ and $W(x)=1$. Then, let us consider
		 
		{\small
			\begin{equation}
				Z(x) = \begin{bmatrix} x_1 & x_2 & x^2_1 & x_1 x_2 & x^2_2 & x^3_1 & x^2_1 x_2 & x_1 x^2_2 & x^3_2 \end{bmatrix}^\top .\nonumber
			\end{equation}
		}
		
		Firstly, we run an experiment with $[t_{I} ,t_{F}]=[0,500]$, $x(t_I ) = \begin{bmatrix} -2 & 2\end{bmatrix}^\top$, $u=1.1sin~(3t)$, and collected the first $T=6000$ samples produced. Notice that those samples taken may not be equally distributed in time. Employing (\ref{eq:c5_1}), the values of $\dot{x}(t)$ were determined for all time instants $\{ t_0 ,t_1 ,\cdots , t_{T-1} \}$. Then, the noise upon each sample of $\dot{x}(t)$ was accounted for according to
		\begin{align}\label{eq:ex1_noise}
			d(t) = \frac{\delta}{\sqrt{n} } \begin{bmatrix} sin(5t) & cos(5t)\end{bmatrix}^\top ,\ \delta = 0.04,
		\end{align}
		which fulfills Assumption \ref{ass:a_2} with $R_D R^\top_D = T \delta^2 I$ , as in \cite{luppi}. Then, $X_1$ is obtained, and matrices $\textbf{A}$, $\textbf{B}$, $\textbf{C}$, $\textbf{Q}$, $\xi$ are determined. 
		
		\subsection*{An Application of Strategy 1 (Algorithm 1)}
		
		In STEP 1, we assumed that $V_0 (x):2-6$, that is, $V_0$ is of degree $2$ to $6$ on $x$. In addition, consider $T_0 (x):2-4$, $\mathcal{Q}_0 (x):2-6$, $u_0 (x):1-3$, $L (x):2$, $\beta_v = \beta_t = \beta_l = 10^{-3}$, $n_v = n_t = n_l = 1$. In STEP 2, $k_{max} = 25$, $\alpha (x):2-8$, and the degrees of $\mathcal{Q}$, $u$, $V$ and $T$ are the same as in STEP 1. In both STEPs, when STOP test (\ref{eq:ox_114}) was employed, the degrees of $V$, $T$ and $\alpha(x)$ also remained unchanged. Then, by applying Algorithm 1, we obtained: 

		{\footnotesize
			\begin{align}
						u (x) = 0.90509 x_1^3 - 0.73433 x_1^2 x_2 + 0.16889 x_1 x_2^2 - 1.316 x_2^3  \nonumber\\
						+ 0.64063 x_1^2 + 0.12002 x_1 x_2 - 0.0031495 x_2^2 + 0.86573 x_1 - 1.9576x_2 . \nonumber
			\end{align}
		}	
		
		In Figure \ref{fig1}, the closed-loop trajectories are portrayed along with the sets \(\mathcal{E}(L,1)\) and \(\mathcal{E}(V,\rho)\) we obtained, with \(\rho = 0.987\).
			\begin{figure}[!ht]
					\centering 
					\includegraphics[scale=0.44]{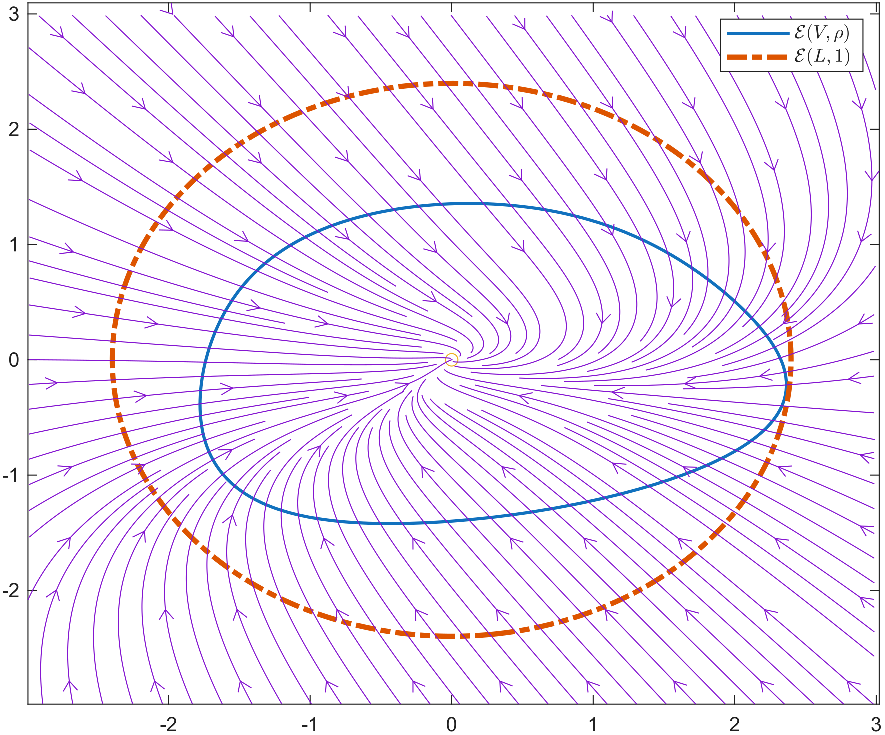}   
					\caption{Closed-loop trajectories produced by Algorithm 1.}
					\label{fig1}
			\end{figure}
				
		The computational complexity was evaluated by the number of decision variables and SOS constraints in the control design program and stop criterion test of each STEP, as well as the total time elapsed. STEP 1 presented 34 decision variables and 7 SOS constraints. STEP 2, 76 decision variables and 7 SOS constraints. The STOP test, 97 decision variables and 7 SOS constraints. The algorithm found a solution within $\kappa=7$ iterations with an execution time of $3.2286$ s.
				
		\subsection*{An Application of Strategy 2 (Algorithm 2)}
				
		In this example, Algorithm 2 was applied to the same system as before and the data was collected performing the same experiment. The input parameters for Algorithm 2 were \(\beta_v = \beta_t = \beta_l = 10^{-6}\), \(n_v = n_t = n_l = 1\), \(\alpha_0 = 10^{-6}||x||^2 I_{1+n+q+N}\), \(s_{vl,0}=s_{tl,0}= 10^{-3}||x||^2\). In STEP 1's control design, we assumed \(V_0(x):2-6\), \(T_0(x):2-4\), \(L(x):2\), \(\mathcal{Q}_0(x):2-6\), \(u_{\gamma,0}(x):1-3\), \(\sigma_0(x):0\), \(\beta_0(x):2\), \(Q_{e,0}(x):2-4\). In STEP 2, \(s_{vl}:2-4\), \(s_{tl}:2-6\), \(\alpha(x):2-8\), \(k_{max} = 15\) and the degrees of \(V\), \(T\), \(\mathcal{Q}\), \(u_\gamma\), \(\sigma\), \(\beta\), \(Q_{e,k}\) remained the same as in STEP 1. In the STOP test, the degrees of \(V\), \(T\), \(\alpha(x)\) remained unchanged. Then, the resulting control law obtained by applying Algorithm 2 was:

		{\footnotesize
				\begin{align}
					u (x) = 0.69063 x_1^3 - 0.59935 x_1^2 x_2 + 1.9954 x_1 x_2^2 - 2.0831 x_2^3 \nonumber \\
					+ 1.452 x_1^2 + 0.41446 x_1 x_2 - 0.60582 x_2^2 + 2.1166 x_1 - 21.7796 x_2 .\nonumber    
				\end{align}
		}
		
		In Figure \ref{fig:ex2_phase}, the close-loop trajectories with control as in (\ref{eq:ox_11}) are portrayed along with the sets \(\mathcal{E}(L,1)\) and \(\mathcal{E}(V,\rho)\), with \(\rho = 0.154\). 
						
			\begin{figure}[!ht]
				\centering
			    \includegraphics[scale=0.44]{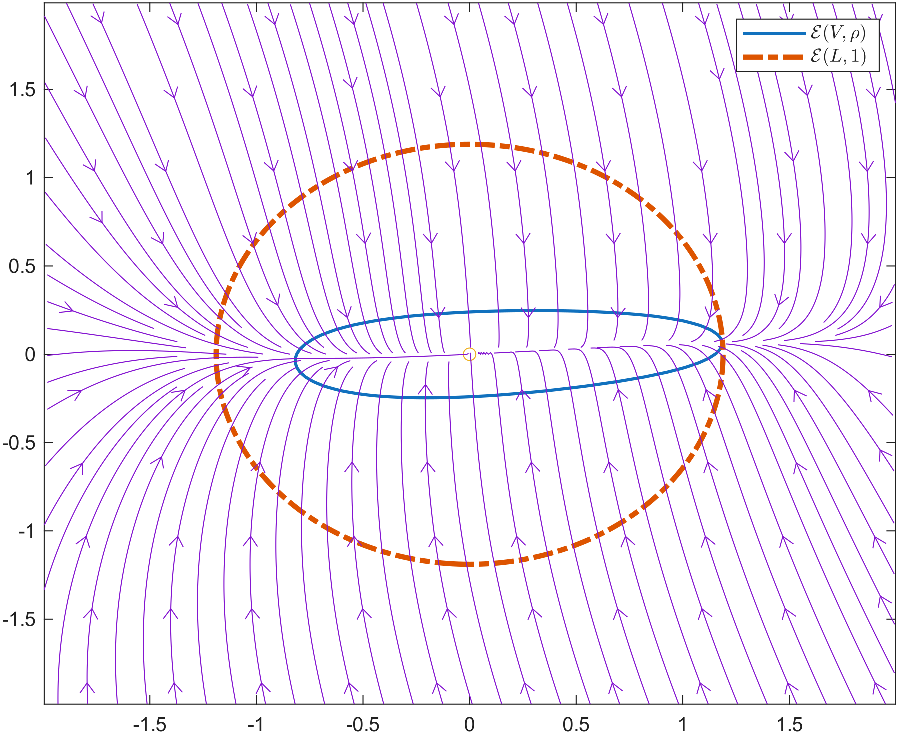}
				\caption{Closed-loop trajectories produced by Algorithm 2.}
				\label{fig:ex2_phase}
			\end{figure}
		Regarding the computational complexity, STEP 1 presented 75 decision variables and 11 SOS constraints. STEP 2, 75 decision variables and 14 SOS constraints. The STOP test, 55 decision variables and 12 SOS constraints. The algorithm found a solution at the first iteration ($\kappa=0$), with an execution time of $2.3103$ s.
						
		In this case, Algorithms 1 and 2 have similar simulation times. However, in general, each repetition of Algorithm 1 is less computationally demanding than those of Algorithm 2. On the other hand, Algorithm 2 typically provides a stabilizing controller at the very first few iterations and is very robust with regards to the data themselves.

		\section{Concluding Remarks}
						
		In this work, new sufficient conditions for state feedback design for nonlinear systems were proposed. This new approach is fully compatible with the well-known framework for data-driven control of polynomials systems available in \cite{biso1} --\cite{persis3}, where a characterization of all systems consistent with the experimental data is given in terms of a matrix ellipsoid. Two iterative SOS strategies for state feedback design were proposed, where stability guarantees for all systems compatible with the data are ensured. We provided computationally efficient approaches that allow to overcome typical drawbacks of other SOS methods in literature, such as the need to employ alternating (D-K) procedures that use model linearization to obtain an initialization for a CLF. Algorithm 1 is of a lower computational complexity than the strategies in \cite{biso1}, \cite{made12}, for example. Algorithm 2, on the other hand, delivers a stabilizing controller in few iterations, at the expense of a higher computational complexity.
						
		Our conditions can also be applied to the case where $d(t)$ is a process disturbance that influences the dynamics, instead of being some noise that appears only upon the measurements of $\dot{x}$. The conditions we have derived remain the same in this case, but the process of data collection would be different. In case $d(t)$ is a process disturbance, it must be applied to the plant in the same sense that the input or excitation signal is. In the case where $d(t)$ is just a noise upon measurements of $\dot{x}$, $d(t)$ is added to $\dot{x}$ only after the excitation $u(t)$ was applied to the plant and the theoretical $x(t)$ (without noise) has been calculated. In this case, the interpretation is that the designed controller $u(x)$ stabilizes all plants $[A~B]\in \mathcal{C}$ when $d=0$.
						
		Future research topics may be feedback control of discrete-time systems, the stabilization of nonzero equilibria, and state feedback design for systems subject to bounds on their state derivatives.


\begin{thebibliography}{00}
							
							\bibitem{hjal}
							H. Hjalmarsson, M. Gevers, S. Gunnarsson, and O. Lequin, ``Iterative feedback tuning: Theory and applications,'' \emph{IEEE Control Syst. Mag.}, vol. 18, no. 4, pp. 26--41, Aug. 1998.
							
							\bibitem{baza1}
							A. Bazanella, L. Campestrini and D. Eckhard, \emph{Data-Driven Controller Design: The $H_2$ Approach.} Haarlem, Netherlands: Springer, 2011.
							
							\bibitem{waarde1}
							H. J. van Waarde, M. K. Camlibel and M. Mesbahi, ``From noisy data to feedback controllers: nonconservative design via a matrix S-lemma,'' \emph{IEEE Trans. Autom. Control}, vol. 67, no. 1, pp. 162--175, Jan. 2022. 
							
							\bibitem{willm3}
							J. C.Willems, P. Rapisarda, I. Markovsky, and B. L. De Moor, ``A note on persistency of excitation, \emph{Syst. Control Lett.}, vol. 54, no. 4, pp. 325--329, Apr. 2005.
							
							\bibitem{morari1}
							M. Tanaskovic, L. Fagiano, C. Novara, and M. Morari, ``Data-driven control of nonlinear systems: An on-line direct approach,'' \emph{Automatica}, vol. 75, pp. 1--10, 2017.
							
							\bibitem{prajna1}
							S. Prajna, A. Papachristodoulou and Fen Wu, ``Nonlinear control synthesis by sum of squares optimization: a Lyapunov-based approach,''  
							in \textit{Proc. 5th Asian Control Conference}, pp. 157-165, 2004.
							
							\bibitem{parri1} 
							P. A. Parrilo, \emph{Structured Semidefinite Programs and Semialgebraic Geometry Methods in Robustness and Optimization.} PhD thesis, California Institute of Technology, Pasadena, CA, 2000.
							
							\bibitem{lasserre}
							D. Henrion and J.-B. Lasserre, ``Convergent relaxations of polynomial matrix inequalities and static output feedback,'' \emph{IEEE Trans. Autom. Control},  vol. 51, no. 2, pp. 192--202, Feb. 2006.
							
							\bibitem{persis1}
							C. De Persis and P. Tesi, ``Formulas for data-driven control: stabilization, optimality, and robustness,'' \emph{IEEE Trans. Autom. Control}, vol. 65, no. 3, pp. 909--924, March 2020. 
							
							\bibitem{persis2}
							M. Guo, C. De Persis and P. Tesi, ``Data-driven stabilization of nonlinear polynomial systems with noisy data,'' \emph{IEEE Trans. Autom. Control}, vol. 67, no. 8, pp. 4210--4217, Aug. 2022.
							
							\bibitem{biso1}
							A. Bisoffi, C. De Persis, and P. Tesi, ``Data-driven control via Petersen's lemma,'' \emph{Automatica}, vol. 145, pp. 110537, 2022.
							
							\bibitem{wilkley2}
							D. de S. Madeira and W. B. Correia, ``Data-Driven Saturated State Feedback Design for Polynomial Systems Using Noisy Data,'' \emph{IEEE Trans. Autom. Control}, vol. 69, no. 11, pp. 7932-7939, Nov. 2024.
							
							\bibitem{persis3}
							A. Luppi, A. Bisoffi, C. De Persis and P. Tesi, ``Data-driven design of safe control for polynomial systems,'' \emph{European Journal of Control}, vol. 75, pp. 100914, 2024.
							
							\bibitem{joao1}
							J. G. N. Silva, D. de S. Madeira and W. B. Correia, ``Data-driven control of polynomial systems using dissipativity-based inequalities and Petersen`s Lemma,'' in \emph{Proc. Int. Conf. Control Mech. Autom. (ICCMA)}, London, UK, 2024, pp. 89--94. 
							
							\bibitem{berbe1}
							J. Berberich, A. Koch, C. W. Scherer, and F. Allg\"ower, ``Robust data-driven state-feedback design,'' in \emph{Proc. Amer. Control Conf.}, Denver, CO, USA, 2020, pp. 1532--1538.
							
							\bibitem{made1}
							D. de S. Madeira, ``Necessary and sufficient dissipativity-based conditions for feedback stabilization,'' 
							\emph{IEEE Trans. Autom. Control}, vol. 67, no. 4, pp. 2100--2107, Apr. 2022.
							
							\bibitem{khali1}
							H. K. Khalil, \emph{Nonlinear Systems.} Upper Saddle River, NJ, USA: Prentice Hall, 2002.
							
							\bibitem{made12} 
							D. de S. Madeira, J. G. N. Silva, G. F. Machado and A. Papachristodoulou, ``Nonlinear static output feedback design for polynomial systems with non-symmetric input saturation bounds,'' \emph{IEEE Control Syst. Lett.}, vol. 9, pp. 937--942, 2025.
							
							\bibitem{bekker}
							P. A. Bekker, ``The Positive Semidefiniteness of Partitioned Matrices'', \textit{Linear Algebra and its Applications}, vol. 111, pp. 261-278, 1988.
							
							\bibitem{prajna2}
							A. Papachristodoulou, J. Anderson, G. Valmorbida, S. Prajna, P. Seiler, P. A. Parrilo, M. M. Peet and D. Jagt, ``{SOSTOOLS}: Sum of squares optimization toolbox for {MATLAB},'' 2021. Available from \texttt{https://github.com/oxfordcontrol/SOSTOOLS}.
								
							\bibitem{DMgithub}
							D. de S. Madeira, J. G. N. Silva, W. B. Correia and A. Papachristodoulou, ``MATLAB Files: (Paper) A New Approach for Feedback Stabilization and its Application for Data-Driven Control of Polynomial Systems'', 2026. Available from \texttt{https://github.com/dmadeiraUFC/Files-TAC-2026}.
							
							\bibitem{moez1}
							M. M. Belhaouane, R. Mtar, H. B. Ayadi, and N. B. Braiek, ``An LMI technique for the global stabilization of nonlinear polynomial systems,'' \emph{Int. J. Comput. Commun. Control.}, vol. 4, no. 4, pp. 335--348, Dec. 2009.
							
							\bibitem{luppi}
							A. Luppi, C. De Persis, and P. Tesi, ``On data-driven stabilization of systems with nonlinearities satisfying quadratic constraints'', \emph{Syst. Control Lett.}, vol. 163, pp. 105206, 2022.
										
	\end{thebibliography}
	\end{document}